\documentclass{article}  
\usepackage{amssymb}              
\usepackage{amsthm}
\usepackage{amsmath}
\usepackage{eucal}
\usepackage{verbatim}
\usepackage[dvips]{graphicx}
\usepackage{multirow}
\usepackage{fancyhdr}
\usepackage{color}
\usepackage{enumerate}
\usepackage{calrsfs}
\usepackage{amsmath}
\usepackage{eufrak}

\usepackage[margin=1.5in]{geometry}

\usepackage{hyperref}

\newtheorem{theorem}{Theorem}
\newtheorem{lemma}{Lemma}

\newtheorem{remark}{Remark}

\makeatletter
\newcommand{\leqnomode}{\tagsleft@true}
\newcommand{\reqnomode}{\tagsleft@false}
\makeatother

\def\({\begin{eqnarray}}
\def\){\end{eqnarray}}
\def\[{\begin{eqnarray*}}
\def\]{\end{eqnarray*}}
\def\part#1#2{\frac{\partial #1}{\partial #2}}

\def\R{\mathbb{R}}
\def\N{\mathbb{N}}
\def\d{\mathrm{d}}
\def\tot#1#2{\frac{\d #1}{\d #2}}

\def\upsi{{\overline\psi}}
\def\uPsi{{\overline\Psi}}

\def\cJH#1{\textcolor{blue}{\bf [#1]}}

\begin{document}

\title{Flocking in the Cucker-Smale model with self-delay and nonsymmetric interaction weights}   
\author{Jan Haskovec\footnote{Computer, Electrical and Mathematical Sciences \& Engineering, King Abdullah University of Science and Technology, 23955 Thuwal, KSA.
jan.haskovec@kaust.edu.sa}}

\date{}

\maketitle

\begin{abstract}
We derive a sufficient condition
for asymptotic flocking in the Cucker-Smale model with self-delay 
(also called reaction delay) and with non-symmetric interaction weights.
The condition prescribes smallness of the delay length
relative to the decay rate of the inter-agent communication weight.
The proof is carried out by a bootstrapping argument combining
a decay estimate for the group velocity diameter
with a variant of the Gronwall-Halanay inequality.
\end{abstract}
\vspace{3mm}

\textbf{Keywords}: Cucker-Smale model, flocking, self-delay, nonsymmetric interaction weights.
\vspace{3mm}

\textbf{2010 MR Subject Classification}: 34K05, 82C22, 34D05, 92D50.
\vspace{3mm}

\section{Introduction}\label{sec:Intro}
In this paper we study the asymptotic behavior
of the Cucker-Smale model \cite{CS1, CS2} with self-delay,
also called reaction-type delay in the previous works \cite{HasMar, HasMar2}.
The Cucker-Smale flocking model is a prototypical model of collective behavior
\cite{Naldi, Vicsek}, describing the dynamics of a group of $N\in\N$ agents,
characterized by their positions $x_i\in\R^d$
and velocities $v_i\in\R^d$,  $i\in\{1,2,\dots,N\}$, with $d\geq 1$.
The agents align their velocities to the average velocity of their conspecifics.
Motivated by applications in biology and social sciences \cite{Camazine, Castellano, Smith} or engineering problems
(for instance, swarm robotics \cite{Hamman, E6, E3B}),
we introduce a fixed time span $\tau>0$ for the agents to process the information
received from their surroundings and take appropriate action.
This leads to the system of delay (functional) differential equations \cite{Smith},
\( 
   \dot x_i(t) &=& v_i(t),   \label{eq:1} \\
   \dot v_i(t) &=& \sum_{j=1}^N \psi_{ij}(t-\tau) (v_j(t-\tau) - v_i(t-\tau)),  \label{eq:2}
\)
for $i\in [N]$, where here and in the sequel we denote $[N] := \{1, 2, \ldots,N\}$.
The system is equipped with the initial datum
\(   \label{IC}
   x_i(t) = x^0_i(t),\qquad
   v_i(t) = v^0_i(t),\qquad
       i\in [N], \quad t \in [-\tau,0],
\)
with prescribed continuous spatial and velocity trajectories $(x^0_i, v^0_i) \in{C}([-\tau,0])$, $i=1,\cdots,N$.
We note that we do not require \eqref{eq:1} to hold for the initial datum.

The communication weights $\psi_{ij}$ in \eqref{eq:2} measure the intensity
of the influence between agents depending on their physical (Euclidean) distance $|x_i - x_j|$.
In the classical setting \cite{CS1, CS2} the communication weights are given by
\(   \label{psi:nonr0}
   \psi_{ij}(t) :=  \frac{1}{N} \psi(|x_j(t) - x_i(t)|),
\)
with the nonnegative, bounded and continuous \emph{influence function} $\psi:[0,\infty)\to [0,\infty)$.
We adopt the assumption that $\psi(s) \leq 1$ for all $s\geq 0$. This, in fact,
can always be achieved by an eventual rescaling of time, and, therefore,
is without loss of generality.

Another form of the communication weights was introduced in \cite{MT},
where the scaling by $1/N$ is replaced by a normalization relative to the influence
of all other agents,
\(\label{psi:r0}
   \psi_{ij}(t) := \frac{\psi(|x_j(t) - x_i(t)|)}{\sum_{\ell=1}^N \psi(|x_\ell(t) - x_i(t)|)}.
\)
Again, the influence function $\psi$ is assumed to be nonnegative and continuous, and verify $\psi\leq 1$ globally.
Note that the normalization in \eqref{psi:r0} leads to nonsymmetric weights,
i.e., in general, $\psi_{ij} \neq \psi_{ji}$.

A generic choice for $\psi$, introduced in \cite{CS1, CS2},
is $\psi(s) = \frac{1}{(1+s^2)^\beta}$ with $\beta>0$. However, we do not restrict ourselves to this particular form in this paper.
Moreover, let us stress that we do \emph{not} impose any symmetry assumptions on the communication weights $\psi_{ij}$,
i.e., we admit $\psi_{ij} \neq \psi_{ji}$ for all $i,j\in [N]$.

As customary in the context of the Cucker-Smale system \cite{CS1, CS2},
we define \emph{(asymptotic) flocking} for solutions of  \eqref{eq:1}--\eqref{eq:2}  as the property
\( \label{def:flocking}
   \sup_{t\geq 0} d_x(t) < \infty, \qquad  \lim_{t\to\infty} d_v(t) = 0,
\)
where the position and, resp., velocity diameters of the agent group $d_x=d_x(t)$ and, resp., $d_v=d_v(t)$ are defined as
\(    \label{def:dxdv}
   d_x(t) := \max_{i,j \in [N]} |x_i(t) - x_j(t)|, \qquad
   d_v(t) := \max_{i,j \in [N]} |v_i(t) - v_j(t)|.
\)
Several previous works focused on derivation of sufficient conditions
for flocking in the Cucker-Smale system with delay.
The papers \cite{Cartabia, ChoiH1, ChoiH2, Choi-Pignotti, Liu-Wu, Pignotti-Reche1, Pignotti-Reche2, Pignotti-Trelat} focus on variants
of the model \emph{without} self-delay (also called propagation- or communication-type delay),
where $v_i$ in \eqref{eq:2} is evaluated at time $t$ instead of $t-\tau$.
This leads to qualitatively different dynamics compared to the system \eqref{eq:1}--\eqref{eq:2}.
In particular, one has an a-priori bound on the velocity radius $R_v(t):= \max_{i\in[N]} |v_i(t)|$
in terms of the initial datum, independently of the delay length $\tau>0$.
In contrast, the model \eqref{eq:1}--\eqref{eq:2} with self-delay exhibits,
for large enough values of $\tau>0$, oscillation of the velocities $v_i=v_i(t)$
with increasing, unbounded amplitude (see Remark \ref{rem:why} below).
In other words, the presence of self-delay fundamentally destabilizes the dynamics
of the system and flocking can only be expected for small enough values of $\tau$.

The Cucker-Smale model with self-delay was studied in \cite{HasMar, HasMar2},
but under the assumption of symmetric communication weights $\psi_{ij}$.
The symmetry leads to conservation of the total momentum $\sum_{i=1}^N v_i$,
and the analysis carried out in \cite{HasMar, HasMar2} is utilizes this fact
in a fundamental way. Finally, \cite{EHS} studies a variant of the model with self-delay and multiplicative noise.
However, the authors impose the assumption of a-priori uniform positivity
of the communication weights $\psi_{ij}$, which goes against
the substance of the Cucker-Smale model (its analysis is interesting precisely
for the fact that the communication weight vanishes at infinity).

The main novelty of this paper is that it provides a sufficient condition
for flocking in the Cucker-Smale system \emph{with} self-delay,
\emph{without} the assumption of symmetric communication weights.
To our best knowledge, no such result exists in the literature.
Our flocking condition is formulated in terms of the delay length $\tau$,
the influence function $\psi$ and the position and velocity diameters
of the initial datum. It is not explicit, however, can be easily inspected numerically.
It will be presented in Section \ref{sec:main} below.

The main difficulties for the analysis stem from the following two properties of system \eqref{eq:1}--\eqref{eq:2}:
non-conservation of the total momentum $\sum_{i=1}^N v_i$ due to
the possible nonsymmetry of the communication weights,
so that the asymptotic flocking vector is not known a-priori;
and, the instability (possible occurrence of unbounded oscillations)
induced by the presence of the self-delay, so that the velocities cannot be bounded a-priori.
These difficulties are overcome by a bootstrapping argument combining
a decay estimate for the group velocity diameter
with a variant of the Gronwall-Halanay inequality,
and will be given in Section \ref{sec:proof}.

\section{Sufficient condition for asymptotic flocking}\label{sec:main}
For the case the influence function $\psi$ was not monotone, let us define its nonincreasing rearrangement
\(  \label{Psi}
   \Psi(u) := \min_{s\in[0,u]} \psi(s) \qquad\mbox{for } u\geq 0.
\)
Moreover, let us denote the initial spatial and velocity diameters
\(  \label{def:dx0dv0}
   d_x^0 := \max_{t\in [-\tau,0]} d_x(t), \qquad
   d_v^0 := \max_{t\in [-\tau,0]} d_v(t),
\)
with $d_x$ and $d_v$ defined in \eqref{def:dxdv}.

\begin{theorem} \label{thm:main}
Let the communication weights $\psi_{ij}$ be given by \eqref{psi:nonr0} or \eqref{psi:r0}
with a nonnegative, bounded and continuous influence function $\psi\leq 1$.
Assume that there exists $C\in (0,1)$ such that
\(     \label{ass:C}
   \Psi\left( d_x^0 + (1+2\tau) \left( \tau + \frac{1}{C}\right) d_v^0 \right) - C \geq 4\tau e^{C\tau} \frac{e^{C\tau}-1}{C\tau},
\)
with $\Psi$ defined in \eqref{Psi} and $d_x^0$, $d_v^0$ given by \eqref{def:dx0dv0}.

Then the system \eqref{eq:1}--\eqref{IC} exhibits flocking in the sense of definition \eqref{def:flocking}.
Moreover, the decay of the velocity diameter is exponential with rate $C$,
\(   \label{concl:thm:main}
   d_v(t) \leq (1+2\tau) d_v^0 \, e^{-C(t-\tau)} \qquad \mbox{for all } t \geq \tau.
\)
\end{theorem}

Condition \eqref{ass:C} is highly nonlinear and, obviously, not verifiable analytically,
apart from trivial cases like $\Psi\equiv 1$.
However, observe that the right-hand side in \eqref{ass:C} is, for a fixed $C>0$, an increasing function of $\tau$,
and vanishing for $\tau\to 0+$. On the other hand, the left-hand side is, for fixed $d_x^0$ and $d_v^0$,
decreasing in $\tau$, and strictly positive for $\tau\to 0+$ if $C>0$ is small enough and $\Psi$ is globally positive
(but does \emph{not} need to be uniformly bounded away from zero).
Consequently, \eqref{ass:C} is to be interpreted, for fixed $d_x^0$ and $d_v^0$, as a smallness condition for $\tau$,
relative to the decay rate of $\Psi$.
Clearly, since $\Psi\leq 1$ by assumption, the necessary condition for \eqref{ass:C} to be verified is $\tau < 1/4$.
In the special case $\Psi\equiv 1$, \eqref{ass:C} is equivalent to $\tau < 1/4$.

A slightly simpler version of \eqref{ass:C} is obtained  for the case
when the initial velocities $v_i^0$ are all constant on $[-\tau,0]$.
Then, by obvious modifications of the steps carried out in Section \ref{sec:proof}, one obtains the following
simplified version of Theorem \ref{thm:main}.

\begin{theorem} \label{thm:main2}
Let the initial velocities $v_i^0$ be all constant on $[-\tau,0]$.
Let the assumptions of Theorem \ref{thm:main} be verified, with \eqref{ass:C} replaced by
\[    
     \Psi\left( d_x^0 + \frac{d_v^0}{C} \right) - C \geq 4\tau e^{C\tau} \frac{e^{C\tau}-1}{C\tau}.
\]
Then the system \eqref{eq:1}--\eqref{IC} exhibits flocking in the sense of definition \eqref{def:flocking}.
Moreover, the decay of the velocity diameter is exponential with rate $C$,
\[
   d_v(t) \leq  d_v^0 \, e^{-C\tau} \qquad \mbox{for all } t \geq 0.
\]
\end{theorem}

\begin{remark}\label{rem:why}
Before we proceed with the proof of Theorem \ref{thm:main}
let us give a short explanation why we cannot expect flocking
to take place in \eqref{eq:1}--\eqref{eq:2} for arbitrary delay lengths $\tau>0$,
even for small initial data.
Indeed, considering the simple case of two agents, $N=2$, in one spatial dimension $d=1$,
with $\psi_{12} \equiv \psi_{21} \equiv 1$, the system \eqref{eq:2} reduces to
\[
    \dot w(t) = -2 w(t - \tau)
\]
for $w:=v_1-v_2$.
Nontrivial solutions of this equation exhibit oscillations whenever $2\tau > e^{-1}$
and their amplitude diverges in time if $2\tau > \pi/2$, see, e.g., \cite{Smith}.
In other words, the system never reaches flocking if $2\tau > \pi/2$,
apart from the trivial case $w\equiv 0$.
\end{remark}


\section{Proof of Theorem \ref{thm:main}}\label{sec:proof}
Let us start by making two simple observations about the communication weights $\psi_{ij}$.
First, due to the assumption $\psi\leq 1$, we have for both the classical \eqref{psi:nonr0} and normalized \eqref{psi:r0} weights
the upper bound
\(   \label{subconv}
   \sum_{i=1}^N \psi_{ij}(t) \leq 1\qquad\mbox{for all } i\in[N] \mbox{ and } t\geq 0.
\)
In fact, for the normalized weights \eqref{psi:r0} the above holds even with equality, but we do not make use of this property in our proof.
Second, due to the assumed continuity of $\psi$, we have the lower bound
\(  \label{Psi:cont}
   \psi_{ij}(t) \geq \frac{\Psi(d_x(t))}{N} \quad \mbox{for all } t\geq 0 \mbox{ and all } i, j\in [N],
\)
with $\Psi$ given by \eqref{Psi}.
Indeed, for \eqref{psi:nonr0} we have
\[
   \psi_{ij}(t) = \frac{1}{N} \psi(|x_i(t)-x_j(t)|) \geq \frac{1}{N}  \Psi(|x_i(t)-x_j(t)|) \geq \frac{\Psi(d_x(t))}{N}.
\]
For \eqref{psi:r0} the same follows due to the assumption $\psi\leq 1$,
\[
    \psi_{ij}(t) \geq \frac{\psi(|x_j(t) - x_i(t)|)}{N} \geq \frac{\Psi(d_x(t))}{N}.
\]

We first prove a result on the decay of the velocity diameter $d_v=d_v(t)$.

\begin{lemma} \label{lem:4}
Let the communication weights $\psi_{ij}$ satisfy \eqref{subconv}
and for a fixed $T>\tau$ denote
\(  \label{ass:lem:4}
   \upsi := N \min_{t\in [\tau,T]} \min_{i\neq j\in [N]} \psi_{ij}(t-\tau).
\)
Then, along the solutions of \eqref{eq:1}--\eqref{eq:2}, 
\(  \label{ineq:react_niw}
   \tot{}{t} d_v (t) \leq 4 \int_{t-\tau}^t d_v(s-\tau) \d s - \upsi d_v (t)
\)
for almost all $t\in (\tau,T)$.
\end{lemma}

\begin{proof}
For the sake of legibility, in this section we shall use the shorthand notation $\widetilde v_j := v_j(t-\tau)$,
while $v_j$ stands for $v_j(t)$.

Due to the continuity of the velocity trajectories $v_i=v_i(t)$,
there is an at most countable system of open, mutually disjoint
intervals $\{\mathcal{I}_\sigma\}_{\sigma\in\N}$ such that
$$
   \bigcup_{\sigma\in\N} \overline{\mathcal{I}_\sigma} = [\tau,\infty)
$$
and for each ${\sigma\in\N}$ there exist indices $i(\sigma)$, $k(\sigma)$
such that
$$
   d_v(t) = |v_{i(\sigma)}(t) - v_{k(\sigma)}(t)| \quad\mbox{for } t\in \mathcal{I}_\sigma.
$$
Then, using the abbreviated notation $i:=i(\sigma)$, $k:=k(\sigma)$,
we have for every $t\in \mathcal{I}_\sigma$,
\(
   \frac12 \tot{}{t} d_v(t)^2 &=& (\dot v_i - \dot v_k)\cdot (v_i-v_k)    \nonumber\\
      &=& 
       \left(\sum_{j=1}^N \psi_{ij}(t) (\widetilde v_j - \widetilde v_i) - \sum_{j=1}^N \psi_{kj}(t) (\widetilde v_j - \widetilde v_k) \right) \cdot (v_i-v_k).
       \label{4:2}
\)
We process the first term of the right-hand side as follows,
\(   \label{4:1}
    \sum_{j=1}^N \psi_{ij}(t) (\widetilde v_j - \widetilde v_i) \cdot (v_i-v_k) &=&
      \sum_{j=1}^N \psi_{ij}(t) (\widetilde v_j - v_j + v_i - \widetilde v_i) \cdot (v_i-v_k) \\
         && + \sum_{j=1}^N \psi_{ij}(t) (v_j- v_i) \cdot (v_i-v_k).  \nonumber
\)
Noting that $t\geq\tau$, we estimate the difference $|\widetilde v_j - v_j|$ by
\[
   |\widetilde v_j - v_j| &\leq& \int_{t-\tau}^t |\dot v_j(s)| \,\d s \\
      &\leq&   
          \int_{t-\tau}^t \sum_{\ell=1}^N \psi_{j\ell}(s) |v_\ell(s-\tau) - v_j(s-\tau)| \,\d s \\
      &\leq&  
          \int_{t-\tau}^t \sum_{\ell=1}^N \psi_{j\ell}(s) \, d_v(s-\tau) \,\d s \\
      &\leq& \int_{t-\tau}^t d_v(s-\tau) \, \d s,
\]
where for the last equality we used the property \eqref{subconv}.

Performing an analogous estimate for the term $|\widetilde v_i - v_i|$
and using the Cauchy-Schwarz inequality and, again,  \eqref{subconv}, we arrive at
\[
   \sum_{j=1}^N \psi_{ij}(t) (\widetilde v_j - v_j + v_i - \widetilde v_i) \cdot (v_i-v_k) &\leq&
      \sum_{j=1}^N \psi_{ij}(t) \left( |\widetilde v_j - v_j| + |\widetilde v_i - v_i| \right) |v_i-v_k| \\
      &\leq&
      2\, d_v(t) \int_{t-\tau}^t d_v(s-\tau) \,\d s.
\]
To estimate the second term of the right-hand side of \eqref{4:1}, observe that, using the Cauchy-Schwarz inequality, we have
\[
   (v_j- v_i)\cdot (v_i-v_k) &=& (v_j-v_k)\cdot(v_i-v_k) - |v_i-v_k|^2 \\
      &\leq& |v_i-v_k| \bigl( |v_j-v_k| - |v_i-v_k| \bigr) \leq 0,
\]
since, by definition, $|v_j-v_k| \leq d_v = |v_i-v_k|$.
Then, with \eqref{ass:lem:4}, we have
\[
   \sum_{j=1}^N \psi_{ij}(t) (v_j- v_i) \cdot (v_i-v_k) \leq
        \frac{\upsi}{N} \sum_{j=1}^N (v_j- v_i) \cdot (v_i-v_k).
\]
Repeating the same steps for the second term of the right-hand side of \eqref{4:2}, we finally arrive at
\[
      \frac12 \tot{}{t} d_v(t)^2 &\leq&
           4\, d_v(t) \int_{t-\tau}^t d_v(s-\tau) \,\d s
         + \frac{\upsi}{N} \left( \sum_{j=1}^N (v_j- v_i) \cdot (v_i-v_k) - \sum_{j=1}^N (v_j- v_k) \cdot (v_i-v_k) \right) \\
         &=&
           4\, d_v(t) \int_{t-\tau}^t d_v(s-\tau) \,\d s - |v_i-v_k|^2 \\
       &\leq&
           4\, d_v(t) \int_{t-\tau}^t d_v(s-\tau) \,\d s - {\upsi}\, d_v(t)^2,
\]
from which \eqref{ineq:react_niw} directly follows, for almost all $t\in (\tau,T)$.
\end{proof}

The proof of Theorem \ref{thm:main} shall be based on the decay estimate of Lemma \ref{lem:4},
combined with the following variant of the Gronwall-Halanay lemma \cite{Halanay}.

 \begin{lemma}\label{lem:Halanay}
 Fix $\tau>0$ and let $u\in{C}([-\tau,\infty))$ be a nonnegative continuous function
with piecewise continuous derivative on $(\tau,\infty)$, such that
for almost all $t>\tau$ the integro-differential inequality is satisfied,
\(   \label{Halanay:1}
   \tot{}{t} u(t) \leq \frac{\alpha}{\tau} \int_{t-\tau}^t u(s-\tau) \, \d s - \beta u(t),
\)
with constants $0 < \alpha < \beta$.
Then there exists a unique $\gamma\in(0,\beta-\alpha)$ such that
\(  \label{Halanay:2}
   \beta - \gamma = \alpha e^{\gamma\tau} \, \frac{e^{\gamma\tau} - 1}{\gamma\tau},
\)
and the estimate holds
\( \label{Halanay:3}
   u(t) \leq \left( \max_{s\in[-\tau,\tau]} u(s) \right) e^{-\gamma(t-\tau)} \qquad \mbox{for all } t \geq \tau.
\)
\end{lemma}

\begin{proof}
The proof is obtained as a slight generalization of \cite[Lemma 2.5]{ChoiH1} and \cite[Lemma 3.3]{H:SIADS}.
\end{proof}

\begin{lemma}
Along the solutions of \eqref{eq:1}--\eqref{eq:2}, we have
\(   \label{bound:tau}
   \max_{s\in[-\tau,\tau]} d_v(s) \leq (1+2\tau) d_v^0,
\)
with $d_v^0$ defined in \eqref{def:dx0dv0}.
\end{lemma}

\begin{proof}
From \eqref{eq:2} we have for all $t\in (0,\tau]$ and $i\in[N]$,
\[
   \left| \dot v_i(t) \right|  \leq \sum_{j=1}^N \psi_{ij}(t-\tau)  \left| v_j(t-\tau) - v_i(t-\tau) \right|
      \leq \sum_{j=1}^N \psi_{ij}(t-\tau)  d_v(t-\tau) 
      \leq d_v^0,
\]
where the last inequality follows from \eqref{subconv}.
Therefore, still for $t\in (0,\tau]$,
\[
   \left| v_i(t) - v_j(t) \right| &\leq& \left| v_i(0) - v_j(0) \right| + \int_0^t \left( \left| \dot v_i(s) \right| + \left| \dot v_j(s) \right| \right) \d s  \\
      &\leq&  d_v(0) + 2\tau d_v^0 \\
      &\leq& (1+2\tau) d_v^0,
\]
and taking a maximum over $i,j\in[N]$ yields \eqref{bound:tau}
\end{proof}

We are now ready to carry out the proof of Theorem \ref{thm:main}.

\begin{proof}
Let $C\in (0,1)$ be given by \eqref{ass:C}.
Due to \eqref{bound:tau} and the continuity of $d_v=d_v(t)$, there exists some $T>\tau$ such that
\(   \label{CSbound}
   \int_\tau^t d_v(s) \d s < \frac{(1+2\tau)d_v^0}{C} \qquad \mbox{for all } t\in (\tau,T).
\)
We claim that $T=\infty$. For contradiction, assume that \eqref{CSbound} holds only until some finite $T>\tau$.
Then we have
\(  \label{forContr}
   \int_\tau^T d_v(s) \d s = \frac{(1+2\tau)d_v^0}{C} . 
\)
By \eqref{eq:1} we readily have
\(   \label{est:dx}
   d_x(t-\tau) \leq d_x^0 + \int_0^{t-\tau} d_v(s) \d s.
\)
The bounds \eqref{bound:tau} and \eqref{CSbound} imply for all $t\in (\tau,T)$,
\[
   \int_0^{t-\tau} d_v(s) \d s &=& \int_0^{\min\{\tau,t-\tau\}} d_v(s) \d s + \int_{\min\{\tau,t-\tau\}}^{t-\tau} d_v(s) \d s \\
      &\leq& (1+2\tau) \min\{\tau,t-\tau\} \d_v^0 + \frac{(1+2\tau)d_v^0}{C} \\
      &\leq& (1+2\tau) \left( \tau + \frac{1}{C}\right) d_v^0.
\]
Consequently,
\(    \label{bound:dx}
    d_x(t-\tau) \leq d_x^0 + (1+2\tau) \left( \tau + \frac{1}{C}\right) d_v^0.
\)
Then, using \eqref{bound:dx} in \eqref{Psi:cont} and recalling that, by definition, $\Psi$ is a nonincreasing function, gives
\(  \label{est:psi}
   \psi_{ij}(t-\tau) \geq \frac{1}{N} \Psi\left( d_x^0 + (1+2\tau) \left( \tau + \frac{1}{C}\right) d_v^0 \right)  
\)
for all $i,j\in[N]$ and $t\in(\tau,T)$.
Therefore, denoting 
$\uPsi_C:=\Psi\left( d_x^0 + (1+2\tau) \left( \tau + \frac{1}{C}\right) d_v^0 \right)$,
we have
\[
   N \min_{t\in [\tau,T]} \min_{i\neq j\in [N]} \psi_{ij}(t-\tau) \geq \uPsi_C
\]
and Lemma \ref{lem:4} gives
\[
   \tot{}{t} d_v (t) \leq 4 \int_{t-\tau}^t d_v(s-\tau) \d s - \uPsi_C d_v(t) \qquad\mbox{for } t\in(\tau,T).
\]
Noting that assumption \eqref{ass:C} with $C\in(0,1)$ implies $\uPsi_C > 4\tau$,
we apply Lemma \ref{lem:Halanay} with $\alpha:=4\tau$ and $\beta:=\uPsi_C$.
This leads to
\(   \label{dv_decay}
   d_v(t) \leq \left( \max_{s\in[-\tau,\tau]} d_v(s) \right) e^{-\gamma(t-\tau)} \leq (1+2\tau) d_v^0 e^{-\gamma(t-\tau)} \qquad \mbox{for } t\in [\tau,T],
\)
with $\gamma\in(0, \uPsi_C - 4\tau)$ the unique solution of
\[
      \uPsi_C - \gamma = 4\tau e^{\gamma\tau} \, \frac{e^{\gamma\tau} - 1}{\gamma\tau}.
\]
Comparing with \eqref{ass:C}, we have
\[
      4\tau e^{C\tau} \, \frac{e^{C\tau}-1}{C\tau} + C \leq 4\tau e^{\gamma\tau} \, \frac{e^{\gamma\tau} - 1}{\gamma\tau} + \gamma,
\]
and the monotonicity of the above expression implies $C\leq \gamma$.
But then \eqref{dv_decay} gives
\[
   \int_\tau^T d_v(t) \d t \leq (1+2\tau) d_v^0 \int_\tau^{T}  e^{-C(t-\tau)} \d t = \frac{(1+2\tau) d_v^0}{C} \left( 1 - e^{-C(T-\tau)} \right) < \frac{(1+2\tau) d_v^0}{C},
\]
which is a contradiction to \eqref{forContr}.
Thus, we conclude that $T=\infty$, i.e., that \eqref{CSbound} holds for all $t>0$.
Consequently, by \eqref{est:dx} we have
\[
   \sup_{t\geq 0} d_x(t) \leq d_x^0 + \int_0^\infty d_v(s) \d s < d_x^0 + (1+2\tau) \left( \tau + \frac{1}{C}\right) d_v^0,
\]
and so the first condition in the definition \ref{def:flocking} of flocking is verified.
Moreover, since \eqref{est:psi} holds for all $t\geq 0$, Lemma \ref{lem:4} with $\upsi:=\upsi_C$ can be applied globally,
and a subsequent application of Lemma \ref{lem:Halanay}
gives the exponential decay of $d_v=d_v(t)$ as claimed by \eqref{concl:thm:main}.
\end{proof}

\begin{remark}
In fact, for our analysis we do not need to restrict the form of the interaction weights $\psi_{ij}$
to either \eqref{psi:nonr0} or \eqref{psi:r0}.
Instead, the proof of Theorem \ref{thm:main} is only based
on the lower \eqref{subconv} and upper \eqref{Psi:cont} bounds on $\psi_{ij}$.
Consequently, its statement remains valid for any form of $\psi_{ij}$,
as long as they verify \eqref{subconv} and \eqref{Psi:cont}.
\end{remark}



\end{document}